\documentclass[reqno]{amsart}

\usepackage{amsmath}
\usepackage{amsfonts}
\usepackage{amssymb,enumerate}
\usepackage{amsthm}
\usepackage[all]{xy}
\usepackage{rotating}
\usepackage{hyperref,color}
\theoremstyle{plain}
\newtheorem{lem}{Lemma}[section]
\newtheorem{cor}[lem]{Corollary}
\newtheorem{prop}[lem]{Proposition}
\newtheorem{thm}[lem]{Theorem}

\theoremstyle{definition}
\newtheorem{defn}[lem]{Definition}
\newtheorem{ex}[lem]{Example}

\newtheorem{disc}[lem]{Remark}

\newtheorem{fact}[lem]{Fact}

\newcommand{\cat}[1]{\mathcal{#1}}

\newcommand{\catd}{\cat{D}}

\newcommand{\pd}{\operatorname{pd}}

\newcommand{\id}{\operatorname{id}}	
\newcommand{\fd}{\operatorname{fd}}

\newcommand{\mspec}{\operatorname{m-Spec}}
\newcommand{\soc}{\operatorname{Soc}}

\newcommand{\HH}{\operatorname{H}}
\newcommand{\Hom}{\operatorname{Hom}}	

\newcommand{\spec}{\operatorname{Spec}}

\newcommand{\im}{\operatorname{Im}}
\newcommand{\shift}{\mathsf{\Sigma}}

\newcommand{\cone}{\operatorname{Cone}}
\newcommand{\Ker}{\operatorname{Ker}}

\newcommand{\ideal}[1]{\mathfrak{#1}}
\newcommand{\m}{\ideal{m}}
\newcommand{\n}{\ideal{n}}
\newcommand{\p}{\ideal{p}}

\newcommand{\fm}{\ideal{m}}

\newcommand{\fa}{\ideal{a}}
\newcommand{\fb}{\ideal{b}}

\newcommand{\comp}[1]{\widehat{#1}}

\newcommand{\supp}{\operatorname{supp}}
\newcommand{\Supp}{\operatorname{Supp}}

\newcommand{\VE}{\operatorname{V}}
\newcommand{\cosupp}{\operatorname{co-supp}}

\newcommand{\bbz}{\mathbb{Z}}

\newcommand{\xra}{\xrightarrow}
\newcommand{\xla}{\xleftarrow}

\newcommand{\x}{\underline{x}}

\renewcommand{\geq}{\geqslant}
\renewcommand{\leq}{\leqslant}

\newcommand{\Ext}[4][R]{\operatorname{Ext}_{#1}^{#2}(#3,#4)}	
\newcommand{\Rhom}[3][R]{\mathbf{R}\!\operatorname{Hom}_{#1}(#2,#3)}	
\newcommand{\Lotimes}[3][R]{#2\otimes^{\mathbf{L}}_{#1}#3}
\newcommand{\Otimes}[3][R]{#2\otimes_{#1}#3}
\renewcommand{\Hom}[3][R]{\operatorname{Hom}_{#1}(#2,#3)}	
\newcommand{\Tor}[4][R]{\operatorname{Tor}^{#1}_{#2}(#3,#4)}

\newcommand{\LL}[2]{\mathbf{L}\Lambda^{\ideal{#1}}(#2)}
\newcommand{\LLL}[2]{\mathbf{L}\widehat\Lambda^{\ideal{#1}}(#2)}
\newcommand{\RG}[2]{\mathbf{R}\Gamma_{\ideal{#1}}(#2)}
\newcommand{\RRG}[2]{\mathbf{R}\widehat\Gamma_{\ideal{#1}}(#2)}

\newcommand{\Comp}[2]{\widehat{#1}^{\ideal{#2}}}

\newcommand{\ssm}{\smallsetminus}

\newcommand{\catdfb}{\catd_{\text{b}}^{\text{f}}}
\newcommand{\catdb}{\catd_{\text{b}}}
\newcommand{\catdf}{\catd^{\text{f}}}
\newcommand{\RGa}[2]{\mathbf{R}\Gamma_{\mathfrak{#1}\Comp R{#1}}(#2)}
\newcommand{\LLa}[2]{\mathbf{L}\Lambda^{\mathfrak{#1}\Comp R{#1}}(#2)}
\newcommand{\LLno}[1]{\mathbf{L}\Lambda^{\ideal{#1}}}
\newcommand{\LLLno}[1]{\mathbf{L}\widehat\Lambda^{\ideal{#1}}}
\newcommand{\RGno}[1]{\mathbf{R}\Gamma_{\ideal{#1}}}
\newcommand{\RRGno}[1]{\mathbf{R}\widehat\Gamma_{\ideal{#1}}}
\newcommand{\RGano}[1]{\mathbf{R}\Gamma_{\mathfrak{#1}\Comp R{#1}}}
\newcommand{\LLano}[1]{\mathbf{L}\Lambda^{\ideal{#1}\Comp R{#1}}}
\newcommand{\fromRG}[2]{\varepsilon_{\ideal{#1}}^{#2}}
\newcommand{\fromRGno}[1]{\varepsilon_{\ideal{#1}}}
\newcommand{\toLL}[2]{\vartheta^{\ideal{#1}}_{#2}}
\newcommand{\toLLno}[1]{\vartheta^{\ideal{#1}}}
\newcommand{\catdator}{\catd(R)_{\text{$\fa$-tor}}}
\newcommand{\catdacomp}{\catd(R)_{\text{$\fa$-comp}}}
\newcommand{\catdaator}{\catd(\Comp Ra)_{\text{$\fa\Comp Ra$-tor}}}
\newcommand{\catdaacomp}{\catd(\Comp Ra)_{\text{$\fa\Comp Ra$-comp}}}
\newcommand{\compsa}{\comp S^{\fa S}}

\numberwithin{equation}{lem}

\begin{document}

\bibliographystyle{amsplain}

\author{Sean Sather-Wagstaff}

\address{Department of Mathematical Sciences,
Clemson University,
O-110 Martin Hall, Box 340975, Clemson, S.C. 29634
USA}

\email{ssather@clemson.edu}

\urladdr{https://ssather.people.clemson.edu/}

\thanks{
Sean Sather-Wagstaff was supported in part by a grant from the NSA}

\author{Richard Wicklein}

\address{Richard Wicklein, Mathematics and Physics Department, MacMurray College, 447 East College Ave., Jacksonville, IL 62650, USA}

\email{richard.wicklein@mac.edu}

\title{Extended Local Cohomology and Local Homology}

%\date{\today}

%\dedicatory{}

\keywords{
Adic finiteness; 
cohomologically cofinite complexes;
derived local cohomology;
derived local homology;
Greenlees-May duality;
MGM equivalence;
support}
\subjclass[2010]{
13B35, % Completion,
13C12, % Torsion modules and ideals,
13D09, % Derived categories,
13D45% % Local cohomology
}

\begin{abstract}
We present an in-depth exploration of the module structures of local (co)homology modules
(moreover, for complexes) over the completion
$\Comp Ra$ of a commutative noetherian ring $R$ with respect to a proper ideal $\fa$. 
In particular, we extend Greenlees-May Duality and MGM Equivalence to track behavior over $\Comp Ra$, 
not just over $R$. 
We apply this to the study of two recent versions of homological finiteness for complexes, 
and to certain isomorphisms, with a view toward further applications.
We also discuss subtleties and simplifications in the computations of these functors. 
%We present an in-depth exploration of the module structures of local (co)homology modules (moreover, for complexes) over the completion $\widehat R^{\mathfrak a}$ of a commutative noetherian ring $R$ with respect to a proper ideal $\mathfrak a$. In particular, we extend Greenlees-May Duality and MGM Equivalence to track behavior over $\widehat R^{\mathfrak a}$, not just over $R$. We apply this to the study of two recent versions of homological finiteness for complexes, and to certain isomorphisms, with a view toward further applications. We also discuss subtleties and simplifications in the computations of these functors. 
% notes: 23 pages. part 4 of a series with http://arxiv.org/abs/1401.6925 and http://arxiv.org/abs/1506.07052. comments welcome
% add other arxiv refs after posting
\end{abstract}

\maketitle

\tableofcontents

\section{Introduction} \label{sec130805a}
Throughout this paper let $R$ be a commutative noetherian ring, let $\fa \subsetneq R$ be a proper ideal of $R$, and let $\Comp{R}{a}$ be the $\fa$-adic completion of $R$.
Let $K$ denote the Koszul complex over $R$ on a finite generating sequence for $\fa$.
We work in the derived category $\catd(R)$ with objects  the $R$-complexes
indexed homologically
$X=\cdots\to X_i\to X_{i-1}\to\cdots$
and the full subcategory $\catdb(R)$ of complexes with bounded homology.
Isomorphisms in $\catd(R)$ are marked by the symbol $\simeq$.
The right derived functor of Hom is $\Rhom --$, and the left derived functor of $\Otimes --$ is $\Lotimes --$.
See, e.g., \cite{hartshorne:rad,verdier:cd,verdier:1} for foundations and Section~\ref{sec140109b} for background.

\

This work is part 4 in a series of papers on derived local cohomology and derived local homology.
It builds on our previous papers~\cite{sather:afbha,sather:afcc,sather:scc}, and it is applied 
in the papers~\cite{sather:afc,sather:asc}.

The starting point for this paper is the following fact. 
Given an $R$-module $M$, each local cohomology module $\HH^i_\fa(M)$ is $\fa$-torsion,
so it has a natural $\Comp Ra$-module structure. 
The completion $\Comp Ma$ also has a natural $\Comp Ra$-module structure. 
More generally, given an $R$-complex $X$, the derived local cohomology complex $\RG aX$
and the derived local homology complex $\LL aX$ are naturally complexes over $\Comp Ra$.
These complexes are constructed by applying the torsion and completion functors, respectively,
to appropriate resolutions of $X$.
For clarity, we write $\RRG aX$ and $\LLL aX$ when we are working over $\Comp Ra$.
See Section~\ref{sec140109b} for  definitions and notation. 
Note that  Section~\ref{sec150922a} documents some subtleties and a simplification involved in these constructions.

In this paper, we investigate how 
standard facts for the $R$-complexes $\RG aX$ and $\LL aX$ extend to
the $\Comp Ra$-complexes $\RRG aX$ and $\LLL aX$.
Our primary motivation comes from work of 
Alonso Tarr{\'{\i}}o, Jerem{\'{\i}}as L{\'o}pez, and Lipman~\cite{lipman:lhcs};
Greenlees and May~\cite{greenlees:dfclh};
Matlis~\cite{matlis:kcd,matlis:hps}; and
Porta, Shaul, and Yekutieli~\cite{yekutieli:hct}.
For instance, the main results of Section~\ref{sec150626a}, summarized next,
extend Greenlees-May Duality and MGM Equivalence (named for Matlis, Greenlees, and May) to this setting.
See Theorems~\ref{thm151002a}, \ref{thm151003a}, and~\ref{thm151003b} in the body of the paper.

\begin{thm}\label{thm151129a}
Let $X,Y\in\catd(R)$ be given.
\begin{enumerate}[\rm(a)]
\item\label{thm151129a1}
there are natural isomorphisms in $\catd(\Comp Ra)$:
\begin{align*}
\Rhom[\Comp Ra]{\LLL aX}{\LLL aY}
&\simeq
\Rhom[\Comp Ra]{\RRG aX}{\RRG aY}\\
&\simeq\Rhom[\Comp Ra]{\RRG aX}{\LLL aY}.
\end{align*}
\item\label{thm151129a2}
The functor
$\RRGno a\colon\catdator\to\catdaator$ is a quasi-equivalence with quasi-inverse given by  
the forgetful functor
$Q\colon \catdaator\to\catdator$.
\item\label{thm151129a3}
The functor
$\LLLno a\colon\catdacomp\to\catdaacomp$  is a quasi-equivalence with quasi-inverse given by 
the forgetful functor
$Q\colon \catdaacomp\to\catdacomp$.
\end{enumerate}
\end{thm}

Here $\catdator$ and $\catdacomp$ are the full subcategories of $\catd(R)$ consisting of the complexes $X$ and $Y$, respectively,
such that the natural morphisms
$\RG aX\to X$ and $Y\to\LL aY$ are isomorphisms.

Section~\ref{sec151003a} investigates the flat and injective dimensions  of the complexes 
$\RRG aX$ and $\LLL aX$ over $\Comp Ra$. In most cases, we bound these above by flat and injective dimensions of $X$ over $R$.

In Section~\ref{sec151002a}, we use these constructions to explain the connection between the
``cohomologically $\fa$-adically cofinite'' complexes of Porta, Shaul, and Yekutieli~\cite{yekutieli:ccc}
and our ``$\fa$-adically finite'' complexes from~\cite{sather:scc}. 
The first of these notions is only defined when $R$ is $\fa$-adically complete;
in this setting, we show that our notion is equivalent; see Proposition~\ref{prop150626a}.
In general, Theorem~\ref{thm150626a} shows that the category of $\fa$-adically finite complexes over $R$
is quasi-equivalent to the category of homologically finite complexes over $\Comp Ra$, hence to the 
category of cohomologically $\fa \Comp Ra$-adically cofinite complexes over $\Comp Ra$.

The concluding Section~\ref{sec151104b} exhibits some isomorphisms for use in~\cite{sather:afc,sather:asc}. 
For instance, the following result is Theorem~\ref{thm151011a} from the body of the paper. 

\begin{thm}\label{thm151129c}
Let $R\to S$ be a  homomorphism of commutative noetherian rings, and let $X\in\catdb(R)$ be  
$\fa$-adically finite over $R$.
If $\Lotimes SX\in\catdb(S)$, e.g., if $\fd_R(S)<\infty$, then there is an isomorphism
in $\catd(\compsa)$:
$$\Lotimes[\Comp Ra]{\compsa}{\LLL aX}\simeq\mathbf{L}\widehat\Lambda^{\fa S}(\Lotimes SX).$$
\end{thm}

When $X$ is homologically finite, this is a straightfoward consequence of the 
isomorphisms $\LLL aX\simeq\Lotimes{\Comp Ra}X$ and 
$\mathbf{L}\widehat\Lambda^{\fa S}(\Lotimes SX)\simeq\Lotimes[S]{\compsa}{(\Lotimes SX)}$.
In general, though, this is more subtle.
And while it may seem esoteric, it is key for understanding some base-change properties in~\cite{sather:asc}.

\section{Background}\label{sec140109b} 

\subsection*{Derived Categories}
In addition to the categories mentioned in Section~\ref{sec130805a},
we also consider the following full triangulated subcategories of $\catd(R)$:

\

$\catd_+(R)$: objects are the complexes $X$ with $\HH_i(X)=0$ for $i\ll 0$.

$\catd_-(R)$: objects are the complexes $X$ with $\HH_i(X)=0$ for $i\gg 0$. 

$\catdf(R)$: objects are the complexes $X$ with $\HH_i(X)$ finitely generated for all $i$.

\

\noindent Doubly ornamented subcategories are defined as intersections, e.g., $\catdf_+(R):=\catdf(R)\bigcap\catd_+(R)$.

\subsection*{Resolutions}
An $R$-complex $F$ is 
\emph{semi-flat}\footnote{In the literature, semi-flat complexes are sometimes called ``K-flat'' or ``DG-flat''.} 
if the functor $\Otimes F-$ respects quasiisomorphisms and each module $F_i$ is flat over $R$, that is,
if $\Otimes F-$ respects injective quasiisomorphisms (see~\cite[1.2.F]{avramov:hdouc}).
A \emph{semi-flat resolution} of an $R$-complex $X$ is a quasiisomorphism $F\xra\simeq X$ such that $F$ is semi-flat.
The \emph{flat dimension} of  $X$ 
$$\fd_R(X):=\inf\{\sup\{i\in\bbz\mid F_i\neq 0\}\mid\text{$F\xra\simeq X$ is a semi-flat resolution}\}$$
is the length of the shortest bounded semi-flat resolution of $X$, if one exists.
The   projective and injective versions (semi-projective, etc.) are defined similarly. 

For the following items, consult~\cite[Section 1]{avramov:hdouc} or~\cite[Chapters 3 and 5]{avramov:dgha}.
Bounded below  complexes of projective $R$-modules are semi-projective, 
bounded below  complexes of flat $R$-modules are semi-flat, and
bounded above  complexes of injective $R$-modules are semi-injective.
Semi-projective $R$-complexes are semi-flat, and every $R$-complex admits a semi-projective resolution (hence, a semi-flat one) and a semi-injective resolution.

\subsection*{Support and Co-support}
The following notions are due to Foxby~\cite{foxby:bcfm} and Benson, Iyengar, and Krause~\cite{benson:csc}.

\begin{defn}\label{defn130503a}
Let $X\in\catd(R)$.
The \emph{small support} and  \emph{small co-support} of $X$ are
\begin{align*}
\operatorname{supp}_R(X)
&=\{\mathfrak{p} \in \operatorname{Spec}(R)\mid \Lotimes{\kappa(\p)}X\not\simeq 0 \} \\
\cosupp_{R}(X)
&=\{\mathfrak{p} \in \operatorname{Spec}(R)\mid \Rhom{\kappa(\p)}X\not\simeq 0 \} 
\end{align*}
where $\kappa(\p):=R_\p/\p R_\p$.
\end{defn}

Much of the following is from~\cite{foxby:bcfm} when $X$ and $Y$ are appropriately bounded and from~\cite{benson:lcstc,benson:csc} in general. 
We refer to~\cite{sather:scc} as a matter of convenience.

\begin{fact}\label{cor130528aw}
Let $X,Y\in\catd(R)$. Then we have $\supp_R(X)=\emptyset$ if and only if $X\simeq 0$ if and only if $\cosupp_R(X)=\emptyset$,
because of~\cite[Fact~3.4 and Proposition~4.7(a)]{sather:scc}.
Also, by~\cite[Propositions~3.12 and~4.10]{sather:scc} we have 
\begin{align*}
\supp_{R}(\Lotimes{X}{Y}) 
&= \supp_R(X)\bigcap\supp_R(Y)\\
\cosupp_{R}(\Rhom{X}{Y}) 
&= \supp_R(X)\bigcap\cosupp_R(Y).
\end{align*}
\end{fact}

\subsection*{Derived Local (Co)homology}
The next notions go back to Grothendieck~\cite{hartshorne:lc}, and Matlis~\cite{matlis:kcd,matlis:hps}, respectively;
see also~\cite{lipman:lhcs,lipman:llcd}.
Let $\Lambda^{\fa}$ denote the $\fa$-adic completion functor, and
$\Gamma_{\fa}$ is the $\fa$-torsion functor, i.e.,
for an $R$-module $M$ we have
$$\Lambda^{\fa}(M)=\Comp Ma
\qquad
\qquad
\qquad
\Gamma_{\fa}(M)=\{ x \in M \mid \fa^{n}x=0 \text{ for } n \gg 0\}.$$ 
A module $M$ is \textit{$\fa$-torsion} if $\Gamma_{\fa}(M)=M$.

The associated left and right derived functors (i.e., \emph{derived local homology and cohomology} functors)
are  $\LL a-$ and $\RG a-$.
Specifically, given an $R$-complex $X\in\catd(R)$ and a semi-flat resolution $F\xra\simeq X$ and a 
semi-injective resolution $X\xra\simeq I$, then we have $\LL aX\simeq\Lambda^{\fa}(F)$ and $\RG aX\simeq\Gamma_{\fa}(I)$.
Note that these definitions yield natural transformations $\RGno a\xra{\fromRGno a}\id\xra{\toLLno a} \LLno a$, induced by the natural morphisms
$\Gamma_{\fa}(I)\xra{\iota_{\fa}^I} I$ and $F\xra{\nu^{\fa}_F} \Lambda^{\fa}(F)$.
Let $\catdator$ denote the full subcategory of $\catd(R)$ of all complexes $X$ such that the morphism
$\RG aX\xra{\fromRG aX}X$ is an isomorphism, and let
$\catdacomp$ denote the full subcategory of $\catd(R)$ of all complexes $Y$ such that the morphism
$Y\xra{\toLL aY}\LL aY$ is an isomorphism.

The definitions of $\RG aX$ and $\LL aX$ yield complexes over the completion $\Comp Ra$, and we denote by
$\LLLno a$ and $\RRGno a$ the associated functors $\catd(R)\to\catd(\Comp Ra)$.
If $Q\colon \catd(\Comp Ra)\to\catd(R)$ is the forgetful functor, then it follows readily that
$Q\circ\LLLno a\simeq\LLno a$ and $Q\circ\RRGno a\simeq\RGno a$.

\begin{fact}\label{fact130619b}
If $X\in\catdf_+(R)$, then there is a natural isomorphism
$\LL aX\simeq \Lotimes{\Comp Ra}{X}$ in $\catd(R)$ by~\cite[Proposition 2.7]{frankild:volh}.
Moreover, the proof of this result shows that there is a natural isomorphism
$\LLL aX\simeq \Lotimes{\Comp Ra}{X}$ in $\catd(\Comp Ra)$.\footnote{This 
is based on the fact that, for a finitely generated free $R$-module $L$, induction on the rank of $L$
shows that the natural isomorphism
$\Otimes{\Comp Ra}{L}\cong\Comp La$ is $\Comp Ra$-linear.} 
By~\cite[Theorem~(0.3) and Corollary~(3.2.5.i)]{lipman:lhcs}, there are natural isomorphisms
of functors
\begin{align*}
\RG a-\simeq\Lotimes{\RG aR}{-}&&
\LL a-\simeq\Rhom{\RG aR}{-}.
\end{align*}
More generally, from~\cite[Theorems~3.2 and~3.6]{shaul:hccac} there are natural isomorphisms of functors $\catd(R)\to\catd(\Comp Ra)$
\begin{gather*}
\RRG a-\simeq\Lotimes{\Comp Ra}{\RG a-}\simeq\Lotimes{\RRG aR}{-}
\\
\LLL a-\simeq\Rhom{\RRG aR}{-}
\simeq\Rhom{\Comp Ra}{\LL a-}.
\end{gather*}
\end{fact}

Here are Greenlees-May duality and MGM equivalence.

\begin{fact}\label{fact150626a}
Given $X,Y\in\catd(R)$, we have natural  isomorphisms in $\catd(R)$
\begin{align*}
\Rhom{\RG aX}{\RG aY}
&\xra\simeq\Rhom{\RG aX}{Y} \\
&\xra\simeq\Rhom{\RG aX}{\LL aY} \\
&\xla\simeq\Rhom{X}{\LL aY} \\
&\xla\simeq\Rhom{\LL aX}{\LL aY}
\end{align*}
induced by $\fromRGno a$ and $\toLLno a$;
see~\cite[Theorem (0.3)$^*$]{lipman:lhcs}.\footnote{See \label{foot151004a}
also~\cite[Thoerem~6.12]{yekutieli:hct}. In addition, 
we have~\cite[Remark~6.14]{yekutieli:hct} for a discussion of some aspects of this result, and~\cite{yekutieli:ehct} for a correction.}, 
From~\cite[Corollary to Theorem~(0.3)$^*$]{lipman:lhcs}
and~\cite[Theorem~1.2]{yekutieli:hct} the next natural morphisms 
are isomorphisms:
\begin{align*}
\RGno a\circ\id\xra[\simeq]{\RGno a\circ\toLLno a}\RGno a\circ\LLno a
&&\LLno a\circ\RGno a\xra[\simeq]{\LLno a\circ\fromRGno a}\LLno a\circ\id \\
\RGno a\circ\RGno a\xra[\simeq]{\fromRGno a\circ\RGno a}\id\circ\RGno a
&&\id\circ\LLno a\xra[\simeq]{\toLLno a\circ\LLno a}\LLno a\circ\LLno a.
\end{align*}
The second row of isomorphisms here shows that
the essential image of $\RGno a$ in $\catd(R)$ is $\catdator$,
and the essential image of $\LLno a$ in $\catd(R)$ is $\catdacomp$.
\end{fact}

\begin{fact}\label{cor130528a}
Let $X\in\catd(R)$. Then 
we know that $\supp_R(X)\subseteq\VE(\fa)$ if and only if $X\in\catdator$ if and only if each homology module $\HH_i(X)$ is $\fa$-torsion,
by~\cite[Proposition~5.4]{sather:scc}
and~\cite[Corollary~4.32]{yekutieli:hct}.\footnote{The affiliated characterization of $\catdacomp$ in terms of co-support is not needed here.}
\end{fact}

\begin{fact}\label{lem151117a}
Let $Y\in\catd(R)$, and consider the following exact triangles in $\catd(R)$.
\begin{align*}
\RG aY\xra{\fromRG aY}Y\to B\to &&
Y\xra{\toLL aY}\LL aY\to C\to 
\end{align*}
By~\cite[Corollary~4.9]{benson:csc} one has 
\begin{gather*}
\supp_R(B)\bigcap\VE(\fa)=\emptyset=\cosupp_R(B)\bigcap\VE(\fa)\\
\supp_R(C)\bigcap\VE(\fa)=\emptyset=\cosupp_R(C)\bigcap\VE(\fa).
\end{gather*}
\end{fact}

\begin{fact}\label{lem150907a}
The following natural transformations 
are isomorphisms
\begin{align*}
\RGno a\circ\RGno a\xra[\simeq]{\RGno a\circ\fromRGno a}\RGno a\circ\id
&&\LLno a\circ\id\xra[\simeq]{\LLno a\circ\toLLno a}\LLno a\circ\LLno a
\end{align*}
by~\cite[Lemma~3.4(a)]{benson:lcstc},
\cite[(4.2)]{benson:csc},
and~\cite[Proposition~3.5.3]{lipman:llcd}.
Note the slight difference between these and the last two isomorphisms in Fact~\ref{fact150626a}.
Note also that one can obtain these isomorphisms as the special case $X=\RG aR$ of the next result.
\end{fact}

\begin{lem}\label{prop151104a}
Let $X\in\catd(R)$ be such that  $\supp_R(X)\subseteq\VE(\fa)$ or $\cosupp_R(X)\subseteq\VE(\fa)$;
e.g.,  $X\simeq K$.
Then the following natural transformations 
are isomorphisms.
\begin{gather*}
\Lotimes X{\RG a-}\xra[\simeq]{\Lotimes X{\fromRGno a}}\Lotimes X{-}\xra[\simeq]{\Lotimes X{\toLLno a}}\Lotimes X{\LL a-}
\\
\Rhom X{\RG a-}\!\xra[\simeq]{\!\Rhom[] X{\fromRGno a}\!}\Rhom{X}{-}\!\xra[\simeq]{\!\Rhom[] X{\toLLno a}\!}\Rhom{X}{\LL a-}
\\
\Rhom {\LL a-}X\!\xra[\simeq]{\!\Rhom[] {\toLLno a}X\!}
\Rhom {-}X
\!\xra[\simeq]{\!\Rhom[] {\fromRGno a}X\!}\Rhom {\RG a-}X
\end{gather*}
\end{lem}

\begin{proof}
Let $Y\in\catd(R)$, and consider the exact triangle
$$\RG aY\xra{\fromRG aY}Y\to B\to$$
in $\catd(R)$, and the following induced triangle.
\begin{equation}\label{eq151117b}
\Lotimes X{\RG aY}\xra{\Lotimes X{\fromRG aY}}\Lotimes X{Y}\to \Lotimes X{B}\to
\end{equation}
Facts~\ref{cor130528a} and~\ref{lem151117a} yield the next sequence
\begin{align*}
\supp_R(\Lotimes XB)
&=\supp_R(X)\bigcap\supp_R(B) 
\subseteq\VE(\fa)\bigcap\supp_R(B) 
=\emptyset.
\end{align*}
We conclude that $\Lotimes XB\simeq 0$, so the exact triangle~\eqref{eq151117b} implies that $\Lotimes X{\fromRG aY}$ is an isomorphism in $\catd(R)$.
The other  isomorphisms from the statement of this result follow similarly.
\end{proof}

\subsection*{Adic Finiteness}
The next two items take their cues from work of 
Hartshorne~\cite{hartshorne:adc},
Kawasaki~\cite{kawasaki:ccma,kawasaki:ccc}, and
Melkersson~\cite{melkersson:mci}.

\begin{fact}[\protect{\cite[Theorem 1.3]{sather:scc}}]
\label{thm130612a}
For $X\in\catd_{\text b}(R)$, the next conditions are equivalent.
\begin{enumerate}[\rm(i)]
\item\label{cor130612a1}
One has $\Lotimes{K^R(\underline{y})}{X}\in\catdfb(R)$  for some (equivalently for every) generating sequence $\underline{y}$ of $\fa$.
\item\label{cor130612a2}
One has  $\Lotimes{X}{R/\mathfrak{a}}\in\catd^{\text{f}}(R)$.
\item\label{cor130612a3}
One has  $\Rhom{R/\mathfrak{a}}{X}\in\catd^{\text{f}}(R)$.
\item\label{cor130612a4}
One has $\LLL aX\in\catdfb(\Comp Ra)$.
\end{enumerate}
\end{fact}

\begin{defn}\label{def120925d}
An $R$-complex $X\in\catdb(R)$ is \emph{$\mathfrak{a}$-adically finite} if it satisfies the equivalent conditions of Fact~\ref{thm130612a} and $\operatorname{supp}_R(X) \subseteq \operatorname{V}(\mathfrak{a})$.
\end{defn}

\begin{ex}\label{ex160206a}
Let $X\in\catdb(R)$ be given.
\begin{enumerate}[(a)]
\item \label{ex160206a1}
If $X\in\catdfb(R)$, then we have $\supp_R(X)=\VE(\fb)$ for some ideal $\fb$, and it follows that $X$ is $\fa$-adically finite
whenever $\fa\subseteq\fb$. (The case $\fa=0$ is from~\cite[Proposition~7.8(a)]{sather:scc}, and the general case follows readily.)
\item \label{ex160206a2}
$K$ and $\RG aR$ are $\fa$-adically finite, by~\cite[Fact~3.4 and Theorem~7.10]{sather:scc}.
\item \label{ex160206a3}
The homology modules of $X$ are artinian if and only if there is an ideal $\fa$ of finite colength (i.e., such that $R/\fa$ is artinian)
such that $X$ is $\fa$-adically finite, by~\cite[Proposition~5.11]{sather:afcc}.
\end{enumerate}
\end{ex}

\section{Computing Derived Functors}\label{sec150922a}

Lipman~\cite[Lemma~3.5.1]{lipman:llcd} shows that, to compute $\RG aX$, one need not use a semi-injective resolution of $X$;
it suffices to use an injective resolution of $X$, i.e., a quasiisomorphism $X\xra\simeq I$ where $I$ is a complex of injective $R$-modules. 
This fact, along with the fact that $\Gamma_{\fa}$ transforms complexes of injective modules into complexes of injective modules,
is the essence of the proof of the first isomorphism in Fact~\ref{lem150907a}.
Our next example shows that~\cite[Lemma~3.5.1]{lipman:llcd} is crucial here, as it shows that $\Gamma_{\fa}$ 
need not respect the class of semi-injective complexes.

\begin{ex}\label{ex150907a}
We consider the following special case of a construction of Chen and Iyengar~\cite[Proposition 2.7]{chen:sirccr}.
Let $k$ be a field, set $R=k[\![X,Y]\!]/(X^2)$, and let $x$ and $y$ denote the residues in $R$ of $X$ and $Y$, respectively.
Set $\n=(x,y)R$ and $\p=xR$, and
consider the injective hull $E=E_R(R/\n)$.
We consider the  complexes
\begin{align*}
F:=&(\cdots\xra xR\xra xR\to 0) & G&:=\bigoplus_{n\in\bbz}\shift^nF \\
I:=&\Hom FE\cong(0\to E\xra xE\xra x\cdots)&J&:=\Hom GE\cong\prod_{n\in\bbz}\shift^nI
\end{align*}
The complex $F$ gives a semi-projective (hence, semi-flat) resolution of $R/xR$, and $I$ yields a semi-injective resolution of $M:=\Hom{R/xR}E$.
Also, $G$ describes a semi-projective (hence semi-flat) resolution of $\bigoplus_{\n\in\bbz}\shift^nR/xR$, and $J$ 
provides a semi-injective resolution of $\prod_{n\in\bbz}\shift^nM$.
Furthermore, $\p$ is an associated prime of each module $J_i\cong\prod_{n\geq i}E$, but $J_\p$ is not semi-injective over $R_\p$:
Chen and Iyengar prove this last claim by noting that $J_\p$ is acyclic  but not contractible over $R_\p$;
it follows that $J_\p$ is acyclic  but not contractible over $R$, so not semi-injective over $R$.

Claim: the complex $\Gamma_{\m}(J)$ is not semi-injective over $R$.
To see this, note that each module $J_i$ is a direct sum of copies of $E$ and copies of $E_R(R/\p)$.
It follows that we have the following natural short exact sequence of complexes:
$$0\to \Gamma_{\m}(J)\to J\to J_\p\to 0.$$
Since $J$ is semi-injective over $R$, the fact that these are complexes of injective $R$-modules implies that $\Gamma_{\m}(J)$ is semi-injective over $R$
if and only if $J_\p$ is so; hence, by the previous paragraph, the claim is established.
\end{ex}

The following lemma is used in the subsequent example.

\begin{lem}\label{lem151001a}
Assume that $(R,\m,k)$ is local and that $\spec(R)=\{\m,\p\}$ with $\p\subsetneq\m$.
Let $L$ be a free $R$-module, and set $L':=\Comp Lm/L$. 
Then $\Comp Lm$ is flat over $R$, and $L'$ is a flat $R_{\p}$-module (hence, $L'$ is also flat over $R$).
If $L$ is not finitely generated over $R$, then $L'\neq 0$.
\end{lem}

\begin{proof}
For the sake of brevity, set $\comp L:=\Comp Lm$.

Krull's Intersection Theorem implies that $L$ is $\m$-adically separated, so the natural map $L\to\comp L$ is injective.
Thus, the definition of $L'$ makes sense.
The module $\comp L$ is flat over $R$
by~\cite[Theorem 5.3.28]{enochs:rha}.
Also, the proof of~\cite[Proposition~6.7.6]{enochs:rha} shows that the inclusion $L\to\comp{L}$ is pure.
It follows that $L'$ is also flat over~$R$.

Claim: $L'$ is naturally an $R_{\p}$-module. 
To see this, first note that~\cite[Corollary~1.9(1)]{yekutieli:ccc} shows that the natural map
$L/\m L\to\comp L/\m\comp L$ is an isomorphism.
Right-exactness of $-\otimes_Rk$ applied to the natural sequence
\begin{equation}\label{eq151001a}
0\to L\xra{\nu^{\m}_L}\comp L\to L'\to 0
\end{equation}
implies that $\Otimes{L'}{k}=0$. 
The fact that $L'$ is flat implies that $\Tor i{L'}{k}=0$ for all $i\neq 0$,
and we conclude that $\m\notin\supp_R(L')$.
It follows from~\cite[Remark~2.9]{foxby:bcfm} that the minimal injective resolution of $L'$ over $R$
contains no summand of the form $E_R(k)$.
Thus, the minimal injective resolution of $L'$ over $R$ has the form
$$0\to E_R(R/\p)^{(\mu^0)}\xra{\partial_0}E_R(R/\p)^{(\mu^1)}\to\cdots.$$
Since each module $E_R(R/\p)^{(\mu^i)}$ is naturally an $R_{\p}$-module, the $R$-module homomorphism
$\partial_0$ is $R_{\p}$-linear. Since $L'$ is isomorphic to $\Ker(\partial_0)$, the claim follows.

We include here a second proof of the claim, as it sheds a different light on the module $L'$, which is somewhat mysterious to us. 
For this second proof, it suffices to show that $L'\cong\Ext 1{R_{\p}}L$, since this Ext-module inherits an $R_{\p}$-structure
from the first slot. (Note that this proof is intimately related to  results of~\cite{yekutieli:hct,yekutieli:sccmc}.)
Since $L$ is flat over $R$, there is an isomorphism
$\LL mL\simeq
\comp L$ in $\catd(R)$.
Thus, the exact sequence~\eqref{eq151001a} provides an exact triangle
\begin{equation}\label{eq151001b}
L\xra{\toLL mL}\LL mL\to L'\to
\end{equation}
in $\catd(R)$. 
Using the structure of $\spec(R)$ again, as in the proof of the claim in Example~\ref{ex150907a}, we have the next exact triangle 
$$\RG mR\xra{\fromRG mR}R\to R_{\p}\to$$
in $\catd(R)$. In the language of~\cite[Section~7]{yekutieli:hct}, this says that we have $\mathbf{R}\Gamma_{0/\m}(R)\simeq R_{\p}$.
The proof of~\cite[Lemma~7.2]{yekutieli:hct} exhibits an exact triangle of the following form.
$$\Rhom{\mathbf{R}\Gamma_{0/\m}(R)}{L}\to L\xra{\toLL mL}\LL mL\to $$
Rotating this triangle, we obtain the next one.
$$L\xra{\toLL mL}\LL mL\to \shift\Rhom{\mathbf{R}\Gamma_{0/\m}(R)}{L}\to$$
Combining this with the triangle in~\eqref{eq151001b}, we conclude that
$$L'\simeq \shift\Rhom{\mathbf{R}\Gamma_{0/\m}(R)}{L}\simeq\shift\Rhom{R_{\p}}L.$$
Applying $\HH_0$, we obtain the next isomorphism
$$L'\cong\HH_0(L')\cong\HH_0(\shift\Rhom{R_{\p}}L)\cong\Ext 1{R_{\p}}L.$$
This concludes the second proof of the claim.

We conclude the proof of the lemma. Since $L'$ is a flat $R$-module, the localization $L'_{\p}\cong L'$ is a flat $R_{\p}$-module;
the isomorphism follows from the above claim.
Finally, assume that $L$ is not finitely generated. To show that $L'\neq 0$, it suffices to show that
$L$ is not  complete. Since $L$ is free of infinite rank, consider a sequence $e_1,e_2,\ldots$ of distinct elements of a basis of $L$.
From our assumption on $\spec(R)$, the nilradical of $R$ is $\p$.
Let $y\in\m\ssm\p$, which is not nilpotent. 
It follows that the Cauchy sequence $\{\sum_{i=1}^ny^ie_i\}_{n=1}^\infty$ in $L$ does not converge in $L$, so $L$ is not complete, as desired.
\end{proof}

Similar to the previous example, the next one shows that $\Lambda^\fa$ does not respect the class of semi-flat complexes. 

\begin{ex}\label{ex150907az}
We continue with the set-up of Example~\ref{ex150907a}. 

Claim: the complex $\Lambda^\fm(G)$ is not semi-flat over $R$.
Since each module $G_i$ is free over $R$, Lemma~\ref{lem151001a} provides a
short exact sequence
$$0\to G\xra{\nu^\fm_G}\Lambda^\fm(G)\to G'\to 0$$
where each module $G'_i$ is a non-zero flat $R_{\p}$-module. Since $R_{\p}$ is Gorenstein and artinian, 
it follows that $G'_i$ is injective over $R_{\p}$, hence 
$$G_i\cong E_{R_{\p}}(\kappa(\p))^{(\mu_i)}\cong E_R(R/\p)^{(\mu_i)}\cong R_{\p}^{(\mu_i)}$$
for some set $\mu_i$. 

Since $G$ is semi-flat, and the modules $\Lambda^\fm(G)$ and $G'_i$ are flat over $R$ for all $i$, 
to establish the claim, it suffices to show that $G'$ is not semi-flat over $R$.
To accomplish this, we follow the lead of Chen and Iyengar~\cite{chen:sirccr}
by showing that $G'$ is exact and minimal.
(Recall that a complex $Z$ is \emph{minimal} if every homotopy equivalence $Z\to Z$ is an isomorphism.)
If $G'$ were semi-flat, this would imply $G'=0$, which we know to be false.

Note that each homology module $\HH_i(G)\cong R/xR$ is $\m$-adically complete, since $R$ is so.
Hence, from~\cite[Theorem~3]{yekutieli:sccmc} we conclude that the natural morphism
$\toLL mG\colon G\to\LL mG$ is an isomorphism in $\catd(R)$.
In other words, the chain map $\nu^\fm_G\colon G\to\Lambda^\fm(G)$ is a quasiisomorphism.
It follows that $G'$ is exact. 

Since $G'$ is a complex of injective $R$-modules, to show that $G'$ is minimal, it suffices to show that
the inclusion $\oplus_{i\in\bbz}\Ker(\partial^{G'}_i)\subseteq\oplus_{i\in\bbz}G'_i$ is an injective envelope over $R$;
see, e.g., \cite[Lemma 5.4.16]{christensen:dcmca}.
As every $R$-module has an injective envelope, we see from~\cite[Corollary~6.4.4]{enochs:rha} that
it suffices to show that each inclusion $\Ker(\partial^{G'}_i)\subseteq G'_i$ is an injective envelope over $R$,
i.e., over $R_\p$ as $G'$ is an $R_{\p}$-complex.

Because $(R_{\p},\p R_{\p},\kappa(\p))$ is a local ring, 
the inclusion $\soc_{R_{\p}}(E_{R_{\p}}(\kappa(\p))^{(\mu_i)})\subseteq E_{R_{\p}}(\kappa(\p))^{(\mu_i)}$ is an injective envelope. 
That is, the inclusion $\soc_{R_{\p}}(G'_i)\subseteq G'_i$ is an injective envelope.
On the other hand, the isomorphism $G'_i\cong R_{\p}^{(\mu_i)}$ works with the conditions
$\p R_{\p}=xR_{\p}\neq 0$ and $x^2=0$ to imply that $\soc_{R_{\p}}(G'_i)=xG'_i=(0:_{G_i}x)$.
Thus, we are reduced to showing that $\Ker(\partial^{G'}_i)=xG'_i$.

By construction, we have $\partial^G_j(G_j)\subseteq xG_{j-1}$ for all $j$,
so $\partial^G_j(xG_j)\subseteq x^2G_{j-1}=0$. In other words,
the composition 
$$G_j\xra{x}G_j\xra{\partial^G_j} G_{j-1}$$ 
is $0$.
Applying the functor $\Lambda^{\m}$, we see that the composition 
$$\Lambda^{\m}(G_j)\xra{x}\Lambda^{\m}(G_j)\xra{\Lambda^{\m}(\partial^G_j)} \Lambda^{\m}(G_{j-1})$$ 
is $0$,
that is, the composition 
$$\Lambda^{\m}(G)_j\xra{x}\Lambda^{\m}(G)_j\xra{\partial^{\Lambda^{\m}(G)}_j} \Lambda^{\m}(G)_{j-1}$$ 
is $0$.
Since $\partial^{G'}$ is induced by $\partial^{\Lambda^{\m}(G)}_j$,
it follows that the composition 
$$G'_j\xra{x}G'_j\xra{\partial^{G'}_j} G'_{j-1}$$ 
is $0$.
We conclude that $0=\partial^{G'}_j(xG'_j)=x\partial^{G'}_j(G'_j)$.
The first equality here implies that $\Ker(\partial^{G'}_i)\supseteq xG'_i$.
Using the second equality here, we see that
$\Ker(\partial^{G'}_i)=\partial^{G'}_{i+1}(G'_{i+1})\subseteq(0:_{G_i}x)=xG'_i$.
We conclude that $\Ker(\partial^{G'}_i)=xG'_i$. This establishes the claim and concludes the example.
\end{ex}

As with~\cite[Lemma~3.5.1]{lipman:llcd}, our next result shows that one can use flat resolutions to compute $\LLno a$.

\begin{prop}
\label{prop160202a}
Let $X\in\catd(R)$, and let $F$ be a complex of flat $R$-modules such that $F\simeq X$ in $\catd(R)$.
Then we have $\LL aX\simeq\Lambda^\fa(F)$ in $\catd(R)$.
\end{prop}

\begin{proof}
Claim: given any exact complex $G$ of flat $R$-modules, one has $\Lambda^\fa(G)\simeq 0$. 
To establish the claim, let $i\in\bbz$ be given;  we need to show that $\HH_i(\Lambda^\fa(G))=0$. 
Assume that the ideal $\fa$ is generated by a sequence of length $n$. 
Then the ``telescope complex'' $T$ is a projective resolution of $\RG aR$ concentrated in degrees $0,\ldots,-n$; see~\cite{yekutieli:hct}. 
From this, we conclude that
$$\HH_{n+1}(\LL aM)\cong\HH_{n+1}(\Rhom{RG aR}M)=0$$
for each $R$-module $M$.

Set $M:=\im(\partial^G_{i-n-1})$, so we have a flat resolution
$$\cdots\xra{\partial^G_{i+1}}G_i\xra{\partial^G_{i}}\cdots\xra{\partial^G_{i-n}}G_{i-n-1}\to M\to 0.$$
From this and the previous paragraph, we conclude that
$$0=\HH_{n+1}(\LL aM)\cong\HH_i(\Lambda^\fa(G)).$$
Since $i$ is arbitrary, this establishes the claim.

Now we prove our proposition. 
Let $P\xra\simeq X$ be a semi-projective resolution. The isomorphism $X\simeq F$ in $\catd(R)$ provides a quasiisomorphism
$\phi\colon P\xra\simeq F$. The mapping cone $G:=\cone(\phi)$ is an exact complex of flat $R$-modules, so the above claim
implies that 
$$0\simeq\Lambda^\fa(G)=\Lambda^\fa(\cone(\phi))\simeq\cone(\Lambda^\fa(\phi)).$$
We conclude that $\Lambda^\fa(\phi)\colon \Lambda^\fa(P)\to\Lambda^\fa(F)$ is a quasiisomorphism,
so $\Lambda^\fa(F)\simeq \Lambda^\fa(P)\simeq\LL aX$, as desired.
\end{proof}

\section{Extended Greenlees-May Duality and MGM Equivalence}\label{sec150626a}

In this section, we extend previous isomorphisms to cover the functors $\LLLno a$ and $\RRGno a$,
beginning with extended versions of parts of Fact~\ref{fact150626a}.

\begin{lem}\label{lem150805a}
The  natural transformations 
$$\LLLno a\circ\RGno a\xra[\simeq]{\LLLno a\circ\fromRGno a}\LLLno a\circ\id\xra[\simeq]{\LLLno a\circ\toLLno a}\LLLno a\circ\LLno a$$
are isomorphisms of functors $\catd(R)\to\catd(\Comp Ra)$.
\end{lem}

\begin{proof}
For the first isomorphism, let $X\in\catd(R)$ be given, and choose semi-projective resolutions $P\xra\simeq X$ and $Q\xra\simeq\RG aX$. 
Let $\phi\colon Q\to P$ be a chain map representing the natural morphism $\RG aX\xra{\fromRG aX} X$.
Then the induced morphism $\LL a{\RG aX}\to\LL aX$ is an isomorphism in $\catd(R)$ by Fact~\ref{fact150626a},
and it is represented by $\Lambda^{\fa}(\phi)\colon\Lambda^{\fa}(Q)\to\Lambda^{\fa}(P)$.
It follows that $\Lambda^{\fa}(\phi)$ is a quasi-isomorphism. 
Since $\Lambda^{\fa}(\phi)$ also represents the natural morphism $\LLL a{\RG aX}\to\LLL aX$ in $\catd(\Comp Ra)$,
this morphism is also an isomorphism, as desired.

For $\LLLno a\circ\toLLno a$, argue similarly with Fact~\ref{lem150907a} in place of Fact~\ref{fact150626a}.
\end{proof}

\begin{lem}\label{lem150805d}
There is a natural isomorphism 
$\LLLno a\circ\RGno a\simeq\LLano a\circ\RRGno a$ of functors $\catd(R)\to\catd(\Comp Ra)$
\end{lem}

\begin{proof}
Let $X\in\catd(R)$ be given, and choose a semi-flat resolution $F\xra\simeq \RRG aX$ over $\Comp Ra$.
Since $\Comp Ra$ is flat over $R$, the complex $F$ is also semi-flat  over $R$,
so it is a semi-flat resolution of $Q(\RRG aX)\simeq\RG aX$ over $R$.
This explains the isomorphisms in $\catd(\Comp Ra)$ in the next display
$$\LLL a{\RG aX}\simeq
\Lambda^{\fa}(F)=\Lambda^{\fa\Comp Ra}(F)\simeq\LLa a{\RRG aX}
$$
The equality comes from the fact that $F$ is an $\Comp Ra$-complex.
\end{proof}

\begin{thm}\label{lem150805e}
The  natural transformation 
$$\id\circ\LLLno a\xra[\simeq]{\vartheta^{\fa\Comp Ra}\circ\LLLno a}
\LLano a\circ\LLLno a$$
is an isomorphism of functors $\catd(R)\to\catd(\Comp Ra)$.
In other words, the essential image of $\LLLno a$ in $\catd(\Comp Ra)$ is contained in $\catd(\Comp Ra)_{\text{$\fa\Comp Ra$}-comp}$.
\end{thm}

\begin{proof}
Let $X\in\catd(R)$ be given.
According to Fact~\ref{fact150626a}, it suffices to show that $\LLL aX$ is of the form $\LLa aY$ for some $Y\in\catd(\Comp Ra)$.
To this end, the next isomorphisms in $\catd(\Comp Ra)$, from Lemmas~\ref{lem150805a} and~\ref{lem150805d}
$$\LLL aX\simeq\LLL a{\RG aX}\simeq \LLa a{\RRG aX}$$
show that the complex $Y=\RRG aX$ satisfies this condition.
\end{proof}

The next few results are proved like the preceding ones, using semi-injective resolutions for the first two.

\begin{lem}\label{lem150805b}
The  natural transformations 
$$\RRGno a\circ\RGno a\xra[\simeq]{\RRGno a\circ\fromRGno a}\RRGno a\circ\id\xra[\simeq]{\RRGno a\circ\toLLno a}\RRGno a\circ\LLno a$$
are isomorphisms of functors $\catd(R)\to\catd(\Comp Ra)$.
\end{lem}

\begin{lem}\label{lem150805c}
There is a natural isomorphism
$\RRGno a\circ\LLno a\simeq\RGano a\circ\LLLno a$ of functors $\catd(R)\to\catd(\Comp Ra)$.
\end{lem}

\begin{thm}\label{lem150805f}
The  natural transformation 
$$
\RGano a\circ\RRGno a
\xra[\simeq]{\varepsilon_{\fa\Comp Ra}\circ\RRGno a}
\id\circ\RRGno a
$$
is an isomorphism of functors $\catd(R)\to\catd(\Comp Ra)$.
In other words, the essential image of $\RRGno a$ in $\catd(\Comp Ra)$ is contained in $\catdaator$.
\end{thm}

\begin{disc}\label{disc160213a}
After we announced the results of this paper, we learned from Liran Shaul that
he has obtained some of the results of this section independently and in more generality.
For instance, the isomorphism $\LLLno a \simeq \LLano a\circ\RRGno a$ from Lemmas~\ref{lem150805a} and~\ref{lem150805d}
is~\cite[Theorem~1.7]{shaul:ard} in a non-noetherian setting. 
\end{disc}

Here is a version of Greenlees-May Duality~\ref{fact150626a} for our extended functors.
It is Theorem~\ref{thm151129a}\eqref{thm151129a1} from the introduction.

\begin{thm}\label{thm151002a}
Given $X,Y\in\catd(R)$,
there are natural isomorphisms in $\catd(\Comp Ra)$:
\begin{align*}
\Rhom[\Comp Ra]{\LLL aX}{\LLL aY}
&\simeq
\Rhom[\Comp Ra]{\RRG aX}{\RRG aY}\\
&\simeq\Rhom[\Comp Ra]{\RRG aX}{\LLL aY}.
\end{align*}
\end{thm}

\begin{proof}
The first isomorphism follows from the next sequence
\begin{align*}
\Rhom[\Comp Ra]{\LLL aX}{\LLL aY}
&\simeq\Rhom[\Comp Ra]{\LLa a{\LLL aX}}{\LLa a{\LLL aY}}
\\
&\simeq\Rhom[\Comp Ra]{\RGa a{\LLL aX}}{\RGa a{\LLL aY}}
\\
&\simeq\Rhom[\Comp Ra]{\RRG a{\LL aX}}{\RRG a{\LL aY}}
\\
&\simeq\Rhom[\Comp Ra]{\RRG aX}{\RRG aY}
\end{align*}
wherein the isomorphisms are from Theorem~\ref{lem150805e}, 
Greenlees-May duality~\ref{fact150626a}, and
Lemma~\ref{lem150805c}, and Lemma~\ref{lem150805b}, respectively.

The second isomorphism follows  from the next sequence
\begin{align*}
\Rhom[\Comp Ra]{\RRG aX}{\RRG aY}
&\simeq\Rhom[\Comp Ra]{\RRG a{\LL aX}}{\RRG a{\LL aY}}
\\
&\simeq\Rhom[\Comp Ra]{\RGa a{\LLL aX}}{\RGa a{\LLL aY}}
\\
&\simeq\Rhom[\Comp Ra]{\RGa a{\LLL aX}}{\LLL aY}
\\
&\simeq\Rhom[\Comp Ra]{\RRG a{\LL aX}}{\LLL aY}
\\
&\simeq\Rhom[\Comp Ra]{\RRG aX}{\LLL aY}
\end{align*}
which is justified similarly.
\end{proof}

\begin{disc}\label{disc151004a}
It is reasonable to ask whether versions of other isomorphisms from Greenlees-May duality~\ref{fact150626a} hold in our set-up. 
For instance, given $X,Y\in\catd(R)$, we have the natural isomorphism
$$\Rhom{\RG aX}{\RG aY}\xra\simeq\Rhom{\RG aX}{Y}.$$
In our set-up, the naive question would ask whether we have
$$\Rhom[\Comp Ra]{\RRG aX}{\RRG aY}\stackrel{\text{?}}\simeq\Rhom[\Comp Ra]{\RRG aX}{Y}$$
in $\catd(\Comp Ra)$. However, this doesn't make sense as $Y$ does not come equipped with an $\Comp Ra$-structure.
On the other hand, see Propositions~\ref{prop151203a}--\ref{cor151004a} and their proofs for some isomorphisms involving $\Rhom[\Comp Ra]{\RRG aX}{Y}$ when $Y\in\catd(\Comp Ra)$.

Another reasonable question to ask would be whether the isomorphism
$$\Rhom[\Comp Ra]{\LLL aX}{\LLL aY}\stackrel{\text{?}}\simeq\Rhom[\Comp Ra]{\LLL aX}{\RRG aY}$$
holds. One sees readily from the example $X=R=Y$ that this fails in general because, one the one hand, we have
$$\Rhom[\Comp Ra]{\LLL aR}{\LLL aR}
\simeq\Rhom[\Comp Ra]{\Comp Ra}{\Comp Ra}\simeq \Comp Ra$$
which is not in general isomorphic to
\begin{align*}
\Rhom[\Comp Ra]{\LLL aX}{\RRG aY}
&\simeq\Rhom[\Comp Ra]{\LLL aR}{\RRG aR}\\
&\simeq\Rhom[\Comp Ra]{\Comp Ra}{\RRG aY}\\
&\simeq\RRG aY.
\end{align*}
\end{disc}

The next two results are akin to~\cite[Corollary~3.9]{shaul:hccac}.

\begin{prop}\label{prop151203a}
Let $X\in\catd(R)$ and $Y\in\catd(\Comp Ra)$ be given.
Then there are isomorphisms in $\catd(\Comp Ra)$
\begin{gather*}
\Rhom[\Comp Ra]{\RRG aX}{Y}\simeq\Rhom[\Comp Ra]{\RRG aX}{\LLa a{Y}} \\
\Rhom[\Comp Ra]{Y}{\LLL aX}\simeq\Rhom[\Comp Ra]{\RGa a{Y}}{\LLL aX} 
\end{gather*}
which are natural in $X$ and $Y$.
\end{prop}

\begin{proof}
We verify the second isomorphism; the verification of the first one is similar.
In the following display, the first step is from Theorem~\ref{lem150805e}:
\begin{align*}
\Rhom[\Comp Ra]{Y}{\LLL aX}
&\simeq\Rhom[\Comp Ra]{Y}{\LLa a{\LLL aX}} \\
&\simeq\Rhom[\Comp Ra]{\RGa a{Y}}{\LLL aX}
\end{align*}
The second step is Greenlees-May duality~\ref{fact150626a}.
\end{proof}

\begin{prop}\label{cor151004a}
Let $X\in\catd(R)$ and $Y\in\catd(\Comp Ra)$ be given, and let $Q\colon \catd(\Comp Ra)\to\catd(R)$ denote the forgetful functor.
Then there are isomorphisms in $\catd(R)$
\begin{gather*}
\Rhom[\Comp Ra]{\RRG aX}{Y}\simeq\Rhom{X}{\LL a{Q(Y)}} \\
\Rhom[\Comp Ra]{Y}{\LLL aX}\simeq\Rhom{\RG a{Q(Y)}}{X} 
\end{gather*}
which are natural in $X$ and $Y$.
\end{prop}

\begin{proof}
We verify the second isomorphism; the verification of the first one is similar.
In the following display, the first step is from Fact~\ref{fact130619b}:
\begin{align*}
\Rhom[\Comp Ra]{Y}{\LLL aX}
&\simeq\Rhom[\Comp Ra]{Y}{\Rhom{\Comp Ra}{\LL aX}} \\
&\simeq\Rhom{\Lotimes[\Comp Ra]{\Comp Ra}{Y}}{\LL aX} \\
&\simeq\Rhom{Q(Y)}{\LL aX} \\
&\simeq\Rhom{\RG a{Q(Y)}}{X}. 
\end{align*}
The second step is Hom-tensor adjointness, and the third one is tensor-cancellation.
The last step is an adjointness isomorphism that follows from 
Fact~\ref{fact130619b}.
\end{proof}

The next two results form our extension of MGM equivalence, as described in parts~\eqref{thm151129a2} and~\eqref{thm151129a3} of
Theorem~\ref{thm151129a} from the introduction; see Remark~\ref{disc151003a}.

\begin{thm}\label{thm151003a}
The functor
$\RRGno a\colon\catdator\to\catdaator$ and 
the forgetful functor
$Q\colon \catdaator\to\catdator$
are quasi-inverse equivalences.
\end{thm}

\begin{proof}
Note that $\RRGno a$ maps $\catdator$ into $\catdaator$ by Theorem~\ref{lem150805f}.
The forgetful functor $Q$ maps $\catdaator$ into $\catdator$ by Fact~\ref{cor130528a} and~\cite[Lemma~5.3]{sather:afcc}.

Next, we show that the composition $\RRGno a\circ Q$ is equivalent to the identity.
For this, consider a complex $X\in\catdaator$;
it suffices to show that $\RRG a{Q(X)}\simeq X$ over $\Comp Ra$.
Fix a semi-injective resolution $X\xra\simeq I$ over $\Comp Ra$.
Since $\Comp Ra$ is flat over $R$, this yields a semi-injective resolution $Q(X)\xra\simeq Q(I)$.
Also, the torsion functors $\Gamma_{\fa}$ and $\Gamma_{\fa\Comp Ra}$ are the same when restricted to $\Comp Ra$-complexes.
By definition of $\catdaator$, the inclusion morphism $\Gamma_{\fa\Comp Ra}(I)\to I$ is a quasiisomorphism.
Thus, we have
$$\RRG a{Q(X)}\simeq\Gamma_{\fa}(Q(I))=\Gamma_{\fa\Comp Ra}(I)\simeq I\simeq X$$ 
over $\Comp Ra$,
hence the desired conclusion.

Lastly, the composition $Q\circ\RRGno a$ is $\RGno a$.
When restricted to $\catdator$, this is  isomorphic to the identity by definition of $\catdator$.
\end{proof}

The next result is proved like the previous one, using a semi-flat resolution.

\begin{thm}\label{thm151003b}
The functor
$\LLLno a\colon\catdacomp\to\catdaacomp$ and 
the forgetful functor
$Q\colon \catdaacomp\to\catdacomp$
are quasi-inverse equivalences.
\end{thm}

The next result is a consequence of the proofs of Theorems~\ref{thm151003a} and~\ref{thm151003b}.

\begin{cor}\label{cor151003a}
There are natural isomorphisms
$Q\circ\RGano a\simeq\RGno a\circ Q$
and 
$Q\circ\LLano a\simeq\LLno a\circ Q$
of functors $\catd(\Comp Ra)\to\catd(R)$.
\end{cor}

\begin{disc}\label{disc151003a}
Theorems~\ref{thm151003a} and~\ref{thm151003b} have several consequences. 
First, they augment Theorems~\ref{lem150805e} and~\ref{lem150805f} by showing that the essential images
of $\RRGno a$ and $\LLLno a$ in $\catd(\Comp Ra)$ are \emph{equal to} $\catdaator$ and $\catdaacomp$, respectively. 

Second, they show that MGM equivalences over $R$ and $\Comp Ra$ are essentially the same.
These equivalences appear in the rows of the following diagram
$$\xymatrix@=15mm{
\catdaator\ar@<1ex>[r]^-{\LLano a}_-{\simeq}\ar@<1.5ex>[d]^-{Q}_-{\simeq}
&\catdaacomp \ar@<1ex>[l]^-{\RGano a}\ar@<1.5ex>[d]^-{Q}_-{\simeq}
\\
\catdator\ar@<1ex>[r]^-{\LLno a}_-{\simeq}\ar@<1.5ex>[u]^-{\RRGno a}
&\catdacomp\ar@<1ex>[l]^-{\RGno a}\ar@<1.5ex>[u]^-{\LLLno a}
}$$
while our results provide the equivalences in the columns. Various versions of this diagram commute.
For instance, the first diagram in the next display commutes by Lemmas~\ref{lem150805a} and~\ref{lem150805d}.
The second one is from Lemmas~\ref{lem150805b} and~\ref{lem150805c}.
$$\xymatrix{
\catdaator\ar[r]^-{\LLano a}_-{\simeq}
&\catdaacomp 
\\
\catdator\ar[r]^-{\LLno a}_-{\simeq}\ar[u]^-{\RRGno a}_-{\simeq}
&\catdacomp\ar[u]_-{\LLLno a}^-{\simeq}}
\qquad
\xymatrix{
\catdaator
&\catdaacomp \ar[l]_-{\RGano a}^-{\simeq}
\\
\catdator\ar[u]^-{\RRGno a}_-{\simeq}
&\catdacomp\ar[l]_-{\RGno a}^-{\simeq}\ar[u]_-{\LLLno a}^-{\simeq}
}$$
Corollary~\ref{cor151003a} explains the commutativity of next two diagrams.
$$\xymatrix{
\catdaator\ar[r]^-{\LLano a}_-{\simeq}\ar[d]_-{Q}^-{\simeq}
&\catdaacomp \ar[d]^-{Q}_-{\simeq}
\\
\catdator\ar[r]^-{\LLno a}_-{\simeq}
&\catdacomp}
\qquad
\xymatrix{
\catdaator\ar[d]_-{Q}^-{\simeq}
&\catdaacomp \ar[l]_-{\RGano a}^-{\simeq}\ar[d]^-{Q}_-{\simeq}
\\
\catdator
&\catdacomp\ar[l]_-{\RGno a}^-{\simeq}
}$$
\end{disc}

\section{Flat and Injective Dimensions}\label{sec151003a}

In this section, we provide bounds on the flat and injective dimensions over $\Comp Ra$
of $\LLL aX$ and $\RRG aX$,
for use in~\cite{sather:asc}.

\begin{prop}\label{cor151012a}
Let $X\in\catdb(R)$ be given.
Then there are inequalities
\begin{align*}
\id_{\Comp Ra}(\LLL aX)&\leq\id_R(X)\\
\fd_{\Comp Ra}(\RRG aX)&\leq\fd_R(X). 
\end{align*}
\end{prop}

\begin{proof}
The \v Cech complex over $\Comp Ra$ on a generating sequence for $\fa$ shows that we have $\fd_{\Comp Ra}(\RRG aR)\leq 0$.
Thus, by the isomorphism $\LLL aX\simeq\Rhom{\RRG aR}{X}$ in $\catd(\Comp Ra)$ from Fact~\ref{fact130619b},
we have 
$$\id_{\Comp Ra}(\LLL aX)=\id_{\Comp Ra}(\Rhom{\RRG aR}{X})\leq \fd_{\Comp Ra}(\RRG aR)+\id_R(X)=\id_R(X)$$ 
by~\cite[Theorem~4.1(F)]{avramov:hdouc}.
Similarly, from the isomorphism $\RRG aX\simeq\Lotimes{\RRG aR}{X}$,
we have $\fd_{\Comp Ra}(\RRG aX)\leq \fd_R(X)$ by~\cite[Theorem~4.1(F)]{avramov:hdouc}.
\end{proof}

We end this section with similar bounds for $\id_{\Comp Ra}(\LLL aX)$ and
$\fd_{\Comp Ra}(\RRG aX)$, after the following lemma.

\begin{lem}\label{lem160205a}
Let $N$ be an $\Comp Ra$-module and is either $\fa$-adically complete or $\fa$-torsion.
Then one has
\begin{align*}
\fd_R(N)&=\fd_{\Comp Ra}(N)
&
\id_R(N)&=\id_{\Comp Ra}(N).
\end{align*}
In particular, $N$ is flat over $R$ if and only if it is flat over $\Comp Ra$,
and $N$ is injective over $R$ if and only if it is injective over $\Comp Ra$.
\end{lem}

\begin{proof}
We verify the first displayed equality in the statement of the lemma; the second one is verified similarly, and the subsequent statements follow from these directly.
Since $\Comp Ra$ is flat over $R$, the inequality $\fd_R(N)\leq\fd_{\Comp Ra}(N)$ is from~\cite[Corollary~4.1(b)(F)]{avramov:hdouc}. (Note that this does not use
the assumption that $N$ is $\fa$-adically complete or $\fa$-torsion.)

Claim. We have
\begin{equation}\label{eq160205a}
\fd_{\Comp Ra}(N)=\sup\{\sup(\Lotimes[\Comp Ra]{(\Comp Ra/P)}{N})\mid P\in\VE(\fa\Comp Ra)\}.
\end{equation}
If $N$ is $\fa$-adically complete, then this is by~\cite[Proposition~2.1]{simon:shpcm}.
If $N$ is $\fa$-torsion, it is straightforward to show that we have $\Supp_{\Comp Ra}(N)\subseteq\VE(\fa\Comp Ra)$, so the claim follows from~\cite[Proposition~5.3.F]{avramov:hdouc}.
(The corresponding formula for injective dimension is from~\cite[Proposition~3.2]{simon:shpcm} and~\cite[Proposition~5.3.I]{avramov:hdouc}.)

Now we complete the proof by verifying the inequality
$\fd_R(N)\geq\fd_{\Comp Ra}(N)$.
The natural isomorphism $\Comp Ra/\fa\Comp Ra\cong R/\fa$ shows that every $P\in\VE(\fa\Comp Ra)$ is of the form $P=\p\Comp Ra$ for 
a unique prime ideal $\p\in\VE(\fa)\subseteq\spec(R)$. This also implies that 
$$\Comp Ra/P=\Comp Ra/\p\Comp Ra\simeq \Lotimes{\Comp Ra}{(R/\p)}$$ so we find that
$$\Lotimes[\Comp Ra]{(\Comp Ra/P)}{N}\simeq
\Lotimes[\Comp Ra]{(\Lotimes{\Comp Ra}{(R/\p)})}{N}\simeq\Lotimes{(R/\p)}N.$$
From this, we have
$$\sup(\Lotimes[\Comp Ra]{(\Comp Ra/P)}{N})=\sup(\Lotimes{(R/\p)}N)\leq\fd_R(N).$$
Thus, the inequality $\fd_R(N)\geq\fd_{\Comp Ra}(N)$ follows from~\eqref{eq160205a}.
\end{proof}

\begin{prop}\label{prop160205a}
Let $X\in\catdb(R)$ be given.
\begin{enumerate}[\rm(a)]
\item\label{prop160205a1}
Then there is an inequality
$\id_{\Comp Ra}(\RRG aX)\leq\id_R(X)$.
\item\label{prop160205a2}
If at least one of the following conditions holds
\begin{enumerate}[\rm \ \ \ \ (1)]
\item $\pd_R(X)<\infty$,
\item $\dim(R)<\infty$, or
\item $X$ is $\fb$-adically finite for some ideal $\fb$, e.g., $X\in\catdfb(R)$,
\end{enumerate}
then one has
$\fd_{\Comp Ra}(\LLL aX)\leq\fd_R(X)$. 
\end{enumerate}
\end{prop}

\begin{proof}
\eqref{prop160205a1}
Assume without loss of generality that $\id_R(X)<\infty$, and let $X\xra\simeq J$ be a bounded semi-injective resolution
over $R$ such that $J_{i}=0$ for all $i<-\id_R(X)$. 
It follows that the $R$-complex $\Gamma_{\ideal{a}}(J)$ is a bounded
semi-injective resolution of $\RG aX$ over $R$ such that 
$\Gamma_{\ideal{a}}(J)_{i}=\Gamma_{\ideal{a}}(J_{i})=0$ for all $i<-\id_R(X)$. 
Since each module in this complex is $\fa$-torsion, the complex $\Gamma_{\ideal{a}}(J)$ is an $\Comp Ra$-complex,
and Lemma~\ref{lem160205a} implies that it consists of injective $\Comp Ra$-modules.
Thus, the $\Comp Ra$-complex $\Gamma_{\ideal{a}}(J)$ is a bounded
semi-injective resolution of $\RRG aX$ over $\Comp Ra$ such that 
$\Gamma_{\ideal{a}}(J)_{i}=\Gamma_{\ideal{a}}(J_{i})=0$ for all $i<-\id_R(X)$. 
The  inequality $\id_{\Comp Ra}(\RRG aX)\leq\id_R(X)$  follows.

\eqref{prop160205a2}
The argument here is similar to the previous one, but with a twist. 
As before, assume without loss of generality that $f=\fd_R(X)<\infty$.
Let $P\xra\simeq X$ be a bounded semi-projective resolution
over $R$. Truncate $P$ appropriately to obtain
a bounded semi-flat resolution $F\xra\simeq X$
over $R$ such that $F_{i}=0$ for all $i>f$ and $F_i=P_i$ for all $i<f$. 

Claim: $\pd_R(F_f)<\infty$. Indeed, in case $\pd_R(X)<\infty$, this is standard.
If $X$ is $\fa$-adically finite, then we have $\pd_R(X)<\infty$ by~\cite[Theorem~6.1]{sather:afcc}, so the claim is established in this case.
Lastly, a result of Gruson and Raynaud~\cite[Seconde Partie, Theorem~3.2.6]{raynaud:cpptpm} and Jensen~\cite[Proposition~6]{jensen:vl} implies that
$\pd_R(F_f)\leq\dim(R)$, so 
the claim holds when $\dim(R)<\infty$.

By the claim, a result of Schoutens~\cite[Theorem~5.9]{schoutens:lfccm} implies that each module $\Lambda^\fa(F_i)$ is flat over $R$.
(If $\dim(R)<\infty$, this is due to Enochs~\cite{enochs:cfm}.)
From Lemma~\ref{lem160205a}, we conclude that each module $\Lambda^\fa(F_i)$ is flat over $\Comp Ra$.
Note that this uses the fact that $\Lambda^\fa(F_i)$ is $\fa$-adically complete; see, e.g., \cite[Corollary~1.7]{yekutieli:fcigm}.
So, 
the $\Comp Ra$-complex $\Lambda^{\ideal{a}}(F)$ is a bounded
semi-flat resolution of $\LLL aX$ over $\Comp Ra$ such that 
$\Lambda^{\ideal{a}}(F)_{i}=0$ for all $i>\fd_R(X)$. 
The  inequality $\fd_{\Comp Ra}(\LLL aX)\leq\fd_R(X)$ now follows.
\end{proof}

\begin{disc}\label{disc160205a}
To our knowledge, it is not known whether the completion $\Lambda^\fa(F)$ of a flat $R$-module $F$ is flat over $R$.
If this is true, then the extra assumptions (1)--(3) can be removed from Proposition~\ref{prop160205a}\eqref{prop160205a2}.
\end{disc}

\section{Cohomological Adic Cofiniteness}\label{sec151002a}

Next, we discuss the connection between $\fa$-adically finite complexes and the following similar notion from~\cite{yekutieli:ccc}.

\begin{defn}\label{defn150626a}
Assume that $R$ is $\fa$-adically complete.
An $R$-complex $X\in\catdb(R)$ is 
\emph{cohomologically $\fa$-adically cofinite} if there is a complex $N\in\catdfb(R)$ such that $X\simeq\RG aN$.
\end{defn}

Our first result in this direction, given next, shows that, when it makes sense to compare these two notions, they are the same.
It is primarily from~\cite{yekutieli:ccc}.

\begin{prop}\label{prop150626a}
Assume that $R$ is $\fa$-adically complete.
An $R$-complex $X\in\catdb(R)$ is cohomologically $\fa$-adically cofinite if and only if it is $\fa$-adically finite.
\end{prop}

\begin{proof}
Assume first that $X$ is cohomologically $\fa$-adically cofinite,
so by definition there is a complex $N\in\catdfb(R)$ such that $X\simeq\RG aN$.
Fact~\ref{fact150626a} implies that $X$ is in $\catdator$, so we have $\supp_R(X)\subseteq\VE(\fa)$
by Fact~\ref{cor130528a}, and we have $\Rhom{R/\fa}X\in\catdf(R)$ by~\cite[Theorem 0.4]{yekutieli:ccc}.
Thus, $X$ is $\fa$-adically finite.

Conversely, assume that $X$ is $\fa$-adically finite.
Then we have $\supp_R(X)\subseteq\VE(\fa)$ by definition, so $X$ is in $\catdator$ by Fact~\ref{cor130528a}.
Also by definition, we have $\Rhom{R/\fa}X\in\catdf(R)$, so according to~\cite[Theorem 0.4]{yekutieli:ccc}, the complex
$X$ is cohomologically $\fa$-adically cofinite.
\end{proof}

The next result gives a similar characterization of $\fa$-adically finite complexes in the incomplete setting.

\begin{thm}\label{thm150626a}
Let $Q\colon\catd(\Comp Ra)\to\catd(R)$ be the forgetful functor,
and let $\catd(R)_{\text{$\fa$-fin}}$ be the full subcategory of $\catd(R)$ consisting of all $\fa$-adically finite $R$-complexes.
\begin{enumerate}[\rm(a)]
\item\label{thm150626a1}
An $R$-complex $X\in\catdb(R)$ is  $\fa$-adically finite if and only if there is a complex $N\in\catdfb(\Comp Ra)$ such that $X\simeq Q(\RGa aN)$.
\item\label{thm150626a2}
The functor $\LLLno a$ induces an equivalence of categories $\catd(R)_{\text{$\fa$-fin}}\to\catdfb(\Comp Ra)$ with quasi-inverse 
induced by $Q\circ\RGano a$.
\item\label{thm150626a3}
If there is a complex $N\in\catdfb(R)$ such that $X\simeq\RG aN$,
then $X\in\catd(R)_{\text{$\fa$-fin}}$.
\end{enumerate}
\end{thm}

\begin{proof}
Claim 1: if $N\in\catdfb(\Comp Ra)$, then we have 
$N\simeq\LLL a{Q(\RGa aN)}$ in $\catd(\Comp Ra)$.
Indeed, the first isomorphism in the following sequence is from Corollary~\ref{cor151003a}.
\begin{align*}
\LLL a{Q(\RGa aN)}
&\simeq\LLL a{\RG a{Q(N)}} 
\simeq\LLL a{Q(N)} 
\simeq N
\end{align*}
The second isomorphism is from Lemma~\ref{lem150805a}.
The third one is from Theorem~\ref{thm151003b}; to apply this result, we use the conditions $N\in\catdfb(\Comp Ra)\subseteq\catdaacomp$.

Claim 2: if $N\in\catdfb(\Comp Ra)$, then we have $Q(\RGa aN)\in \catd(R)_{\text{$\fa$-fin}}$.
Indeed, Corollary~\ref{cor151003a} implies that $Q(\RGa aN)\simeq \RG a{Q(N)}$, so we have
$$\supp_R(Q(\RGa aN))=\supp_R(\RG a{Q(N)})\subseteq\VE(\fa)$$ 
by~\cite[Proposition~3.6]{sather:scc}.
Also, Claim~1 shows that $\LLL a{Q(\RGa aN)}\simeq N\in\catdfb(\Comp Ra)$, so $Q(\RGa aN)$
is $\fa$-adically finite by definition.

Claim 3: if $X\in\catd(R)_{\text{$\fa$-fin}}$, then 
$X\simeq Q(\mathbf{R}\Gamma_{\fa\Comp Ra}(\LLL aX))$
in $\catd(R)$.
The first three isomorphisms in the following sequence are from 
Lemma~\ref{lem150805c}, Lemma~\ref{lem150805b}, and Theorem~\ref{thm151003a}, respectively.
\begin{align*}
Q(\mathbf{R}\Gamma_{\fa\Comp Ra}(\LLL aX))
&\simeq Q(\RRG a{\LL aX})
\simeq Q(\RRG a{\RG aX})
\simeq\RG aX
\simeq X
\end{align*}
The fourth isomorphism is by Fact~\ref{cor130528a}, as we have $\supp_R(X)\subseteq\VE(\fa)$ by assumption.

Now we complete the proof of the result. 
By definition, if $X\in\catd(R)_{\text{$\fa$-fin}}$, then $\LLL aX\in\catdfb(\Comp Ra)$, 
that is, the functor $\LLLno a$ maps $\catd(R)_{\text{$\fa$-fin}}$ to $\catdfb(\Comp Ra)$.
Claim~2 shows that $Q\circ\RGano a$ maps $\catdfb(\Comp Ra)$ to $\catd(R)_{\text{$\fa$-fin}}$, 
and Claim 1 shows that the composition $\LLLno a\circ Q\circ\RGano a$ is isomorphic to the identity on $\catdfb(\Comp Ra)$.
Claim~3 shows that the composition $Q\circ\RGano a\circ\LLLno a$ is isomorphic to the identity on $\catd(R)_{\text{$\fa$-fin}}$.
This establishes part~\eqref{thm150626a2} of the theorem, and part~\eqref{thm150626a1} follows.

For part~\eqref{thm150626a3}, assume that there is a complex $N\in\catdfb(R)$ such that $X\simeq\RG aN$.
Then the complex $N':=\LLL aN\simeq\Lotimes{\Comp Ra}N\in\catdfb(\Comp Ra)$
satisfies
\begin{align*}
Q(\RGa a{N'})
&\simeq Q(\RGa a{\LLL aN})\\
&\simeq\RG a{Q(\LLL aN)}\\
&\simeq\RG a{\LL aN}\\
&\simeq\RG aN\\
&\simeq X;
\end{align*}
see Fact~\ref{fact130619b}.
So, we have $X\simeq Q(\RGa a{N'})\in\catd(R)_{\text{$\fa$-fin}}$
by part~\eqref{thm150626a1}.
\end{proof}

The next example shows that the converse of Theorem~\ref{thm150626a}\eqref{thm150626a3} fails in general.
Thus, the characterization in Theorem~\ref{thm150626a}\eqref{thm150626a1} cannot be simplified (at least not in the naive manner suggested
by Theorem~\ref{thm150626a}\eqref{thm150626a3}).

\begin{ex}\label{ex150628a}
Let $(R,\m,k)$ be a local ring that does not admit a dualizing complex.
Such a ring exists by work of Ogoma~\cite{ogoma}.
Set $\comp R:=\Comp Rm$.
The injective hull $E:=E_R(k)$ is $\m$-adically finite by~\cite[Proposition~7.8(b)]{sather:scc}.
Suppose that there is an $R$-complex $N\in\catdfb(R)$ such that $E\simeq\RG mN$. 
In~\cite[Example~6.7]{sather:asc} we show that this implies that $N$ is dualizing for $R$, contradicting our assumption on $R$.
\end{ex}

\section{Induced Isomorphisms}\label{sec151104b}

This section consists of useful isomorphisms derived from our preceding  results.
We begin with extended versions of Lemma~\ref{prop151104a}.

\begin{prop}\label{prop151105a}
Let $X\in\catd(\Comp Ra)$ be such that $\supp_{\Comp Ra}(X)\subseteq\VE(\fa\Comp Ra)$.
Given an $R$-complex $M\in\catd(R)$, there are isomorphisms in $\catd(\Comp Ra)$
\begin{gather*}
\Lotimes[\Comp Ra] X{\RRG aM}\simeq\Lotimes XM\simeq
\Lotimes[\Comp Ra] X{(\Lotimes{\Comp Ra}M)}
\\
\Rhom[\Comp Ra]{X}{\LLL aM}\simeq\Rhom XM\simeq\Rhom[\Comp Ra]{X}{\Rhom{\Comp Ra}{M}}.
\end{gather*}
\end{prop}

\begin{proof}
We verify the first two isomorphisms; the others are verified similarly.
The first isomorphism in $\catd(\Comp Ra)$ in
the following sequence is from Fact~\ref{fact130619b}.
\begin{align*}
\Lotimes[\Comp Ra] X{\RRG aM}
&\simeq\Lotimes[\Comp Ra] X{(\Lotimes{\Comp Ra}{\RG aM})}
\\
&\simeq\Lotimes X{\RG aM}
\\
&\simeq\Lotimes XM \\
&\simeq\Lotimes[\Comp Ra] X{(\Lotimes{\Comp Ra}M)}
\end{align*}
The second and fourth ones are tensor-cancellation.
The third one is the natural one $\Lotimes X{\RG aM}\xra{\Lotimes X{\fromRG aM}}\Lotimes XM$;
this is an isomorphism in $\catd(R)$ by Lemma~\ref{prop151104a}, and it respects the $\Comp Ra$-structure coming from the left.
\end{proof}

\begin{cor}\label{cor151105a}
Let $\comp K$ be the Koszul complex over $\Comp Ra$ on the generating sequence $\x$ for $\fa$.
Given an $R$-complex $M\in\catd(R)$, there are isomorphisms in $\catd(\Comp Ra)$
\begin{gather*}
\Lotimes[\Comp Ra] {\comp K}{\RRG aM}\simeq
\Lotimes[\Comp Ra] {\comp K}{(\Lotimes{\Comp Ra}M)}\simeq\Lotimes[\Comp Ra] {\comp K}{\LLL aM}
\\
\Rhom[\Comp Ra]{{\comp K}}{\LLL aM}\!\simeq\Rhom[\Comp Ra]{{\comp K}}{\Rhom{\Comp Ra}{M}}\simeq\Rhom[\Comp Ra]{{\comp K}}{\RRG aM}.
\end{gather*}
\end{cor}

The next result is Theorem~\ref{thm151129c} from the introduction.
Note that it is straightforward when $X\in\catdfb(R)$, by Fact~\ref{fact130619b}.
See~\cite[Theorems~5.6 and~5.7]{sather:asc} for applications.

\begin{thm}\label{thm151011a}
Let $R\to S$ be a  homomorphism of commutative noetherian rings, and let $X\in\catdb(R)$ be  
$\fa$-adically finite over $R$.
If $\Lotimes SX\in\catdb(S)$, e.g., if $\fd_R(S)<\infty$,
then there is an isomorphism
in $\catd(\compsa)$:
$$\Lotimes[\Comp Ra]{\compsa}{\LLL aX}\simeq\mathbf{L}\widehat\Lambda^{\fa S}(\Lotimes SX).$$
\end{thm}

\begin{proof}
Our finiteness assumption on $X$ implies that $X\in\catdb(R)$ and $\LLL aX\in\catdfb(\Comp Ra)$.
Thus, we have
$\Lotimes[\Comp Ra]{\compsa}{\LLL aX}\in\catdf_+(\compsa)$.
Also, from~\cite[Theorem~5.10]{sather:afcc}, we have $\mathbf{L}\widehat\Lambda^{\fa S}(\Lotimes SX)\in\catdfb(\compsa)$.
For clarity, we set $K^R:=K=K^R(\x)$,
where $\x$ is a finite generating sequence for $\fa$, and set $K^{\Comp Ra}:=K^{\Comp Ra}(\x)$, and similarly for $K^S$ and $K^{\compsa}$. 

Claim 1: there is an isomorphism 
$\Lotimes[\Comp Ra]{K^{\Comp Ra}}{\LLL aX}\simeq\Lotimes{\Comp Ra}{(\Lotimes{K^R}X)}$ in $\catd(\Comp Ra)$.
This follows from the next sequence of isomorphisms:
\begin{align*}
\Lotimes[\Comp Ra]{K^{\Comp Ra}}{\LLL aX}
&\simeq \Lotimes[\Comp Ra]{K^{\Comp Ra}}{(\Lotimes{\Comp Ra}X)}\\
&\simeq \Lotimes[\Comp Ra]{(\Lotimes {\Comp Ra}{K^R})}{(\Lotimes{\Comp Ra}X)}\\
&\simeq\Lotimes{\Comp Ra}{(\Lotimes{K^R}X)}.
\end{align*}
The first  isomorphism is from Corollary~\ref{cor151105a},
and the others are routine.

Claim 2: we have $\Lotimes[\Comp Ra]{\compsa}{\LLL aX}\in\catdfb(\compsa)$.
To this end, recall that the first paragraph of this proof shows that we have 
$\Lotimes[\Comp Ra]{\compsa}{\LLL aX}\in\catdf_+(\compsa)$.
Thus, we need only show that $\Lotimes[\Comp Ra]{\compsa}{\LLL aX}\in\catdb(\compsa)$.
For this, note first that every maximal ideal of $\compsa$ contains $\fa\compsa$.
In other words, we have $\supp_{\compsa}(K^{\compsa})=\VE(\fa\compsa)\supseteq\mspec(\compsa)$.
Thus, according to~\cite[Theorem~4.2(b)]{frankild:rrhffd},
to establish the claim, it suffices to show that we have
$\Lotimes[\compsa]{K^{\compsa}}{(\Lotimes[\Comp Ra]{\compsa}{\LLL aX})}\in\catdfb(\compsa)$.
To this end, we consider the following sequence of isomorphisms in $\catd(\compsa)$.
\begin{align*}
\Lotimes[\compsa]{K^{\compsa}}{(\Lotimes[\Comp Ra]{\compsa}{\LLL aX})}
&\simeq\Lotimes[\compsa]{(\Lotimes[\Comp Ra]{\compsa}K^{\Comp Ra})}{(\Lotimes[\Comp Ra]{\compsa}{\LLL aX})}
\\
&\simeq
\Lotimes[\Comp Ra]{\compsa}{(\Lotimes[\Comp Ra]{K^{\Comp Ra}}{\LLL aX})}
\\
&\simeq
\Lotimes[\Comp Ra]{\compsa}{(\Lotimes{\Comp Ra}{(\Lotimes{K^R}X)})}
\\
&\simeq
\Lotimes[S]{\compsa}{(\Lotimes{S}{(\Lotimes{K^R}X)})}
\\
&\simeq
\Lotimes[S]{\compsa}{(\Lotimes[S]{(\Lotimes{S}{K^R})}{(\Lotimes{S}X)})}
\\
&\simeq
\Lotimes[S]{\compsa}{(\Lotimes[S]{K^S}{(\Lotimes{S}X)})}
\end{align*} 
The third isomorphism is from Claim 1, and the others are standard.
Since we have $\Lotimes{S}X\in\catdb(S)$, by assumption,
the condition $\pd_S(K^S)<\infty$ implies that $\Lotimes[S]{K^S}{(\Lotimes{S}X)}\in\catdb(S)$.
Thus the flatness of $\compsa$ over $S$ implies that we have
$$\Lotimes[\compsa]{K^{\compsa}}{(\Lotimes[\Comp Ra]{\compsa}{\LLL aX})}
\simeq\Lotimes[S]{\compsa}{(\Lotimes[S]{K^S}{(\Lotimes{S}X)})}\in\catdb(\compsa).$$
This establishes Claim 2.

Since the complexes 
$\mathbf{L}\widehat\Lambda^{\fa S}(\Lotimes SX)$ and
$\Lotimes[\Comp Ra]{\compsa}{\LLL aX}$
are in $\catdfb(\compsa)$, to show that they are isomorphic, Theorem~\ref{thm150626a}\eqref{thm150626a2}
says that it suffices to show that
$\mathbf{R}\Gamma_{\fa \compsa}(\mathbf{L}\widehat\Lambda^{\fa S}(\Lotimes SX))
\simeq\mathbf{R}\Gamma_{\fa \compsa}(\Lotimes[\Comp Ra]{\compsa}{\LLL aX})$
in $\catd(\compsa)$.
To verify this isomorphism, we compute as follows.
\begin{align*}
\mathbf{R}\Gamma_{\fa \compsa}(\mathbf{L}\widehat\Lambda^{\fa S}(\Lotimes SX))
&\simeq \mathbf{R}\widehat\Gamma_{\fa S}(\Lotimes SX)
\\
&\simeq \Lotimes[S]{\compsa}{\mathbf{R}\Gamma_{\fa S}(\Lotimes SX)}
\\
&\simeq \Lotimes[S]{\compsa}{(\Lotimes SX)}
\\
&\simeq \Lotimes{\compsa}{X}
\end{align*}
The first isomorphism is by Lemmas~\ref{lem150805b} and~\ref{lem150805c}
The second isomorphism is from Fact~\ref{fact130619b}.
For the third isomorphism, note that~\cite[Lemma~5.7]{sather:afcc} shows that $\supp_S(\Lotimes SX)\subseteq\VE(\fa S)$,
so Fact~\ref{cor130528a} implies that $\mathbf{R}\Gamma_{\fa S}(\Lotimes SX)\simeq\Lotimes SX$ in $\catd(S)$.
The last isomorphism is tensor-cancellation.
The next isomorphisms are justified similarly.
\begin{align*}
\mathbf{R}\Gamma_{\fa \compsa}(\Lotimes[\Comp Ra]{\compsa}{\LLL aX})
&\simeq\Lotimes[\compsa]{\mathbf{R}\Gamma_{\fa \compsa}(\compsa)}{(\Lotimes[\Comp Ra]{\compsa}{\LLL aX})}
\\
&\simeq\Lotimes[\compsa]{(\Lotimes[\Comp Ra]{\compsa}{\mathbf{R}\Gamma_{\fa \Comp Ra}(\Comp Ra)})}{(\Lotimes[\Comp Ra]{\compsa}{\LLL aX})}
\\
&\simeq\Lotimes[\Comp Ra]{\compsa}{(\Lotimes[\Comp Ra]{\mathbf{R}\Gamma_{\fa \Comp Ra}(\Comp Ra)}{\LLL aX})}
\\
&\simeq\Lotimes[\Comp Ra]{\compsa}{\mathbf{R}\Gamma_{\fa \Comp Ra}(\LLL aX)}
\\
&\simeq\Lotimes[\Comp Ra]{\compsa}{\RRG aX}
\\
&\simeq\Lotimes[\Comp Ra]{\compsa}{(\Lotimes{\Comp Ra}{\RG aX})}
\\
&\simeq\Lotimes[\Comp Ra]{\compsa}{(\Lotimes{\Comp Ra}{X})}
\\
&\simeq\Lotimes{\compsa}{X}
\end{align*}
These two sequences give the desired isomorphism, completing the proof.
\end{proof}

\section*{Acknowledgments}
We are grateful to Srikanth Iyengar, 
Liran Shaul,
and Amnon Yekutieli
for helpful comments about this work.

%\bibliography{../+new}

\begin{thebibliography}{10}

\bibitem{lipman:lhcs}
L.\ {Alonso Tarr{\'{\i}}o}, A.~Jerem{\'{\i}}as L{\'o}pez, and J.\ Lipman,
  \emph{Local homology and cohomology on schemes}, Ann. Sci. \'Ecole Norm. Sup.
  (4) \textbf{30} (1997), no.~1, 1--39. \MR{1422312 (98d:14028)}

\bibitem{avramov:hdouc}
L.~L. Avramov and H.-B.\ Foxby, \emph{Homological dimensions of unbounded
  complexes}, J. Pure Appl. Algebra \textbf{71} (1991), 129--155.
  \MR{93g:18017}

\bibitem{avramov:dgha}
L.~L. Avramov, H.-B.\ Foxby, and S.\ Halperin, \emph{Differential graded
  homological algebra}, in preparation.

\bibitem{benson:lcstc}
D.~Benson, S.~B. Iyengar, and H.~Krause, \emph{Local cohomology and support for
  triangulated categories}, Ann. Sci. \'Ec. Norm. Sup\'er. (4) \textbf{41}
  (2008), no.~4, 573--619. \MR{2489634 (2009k:18012)}

\bibitem{benson:csc}
\bysame, \emph{Colocalizing subcategories and cosupport}, J. Reine Angew. Math.
  \textbf{673} (2012), 161--207. \MR{2999131}

\bibitem{chen:sirccr}
X.-W. Chen and S.~B. Iyengar, \emph{Support and injective resolutions of
  complexes over commutative rings}, Homology, Homotopy Appl. \textbf{12}
  (2010), no.~1, 39--44. \MR{2594681 (2011b:13035)}

\bibitem{enochs:cfm}
E.~E. Enochs, \emph{Complete flat modules}, Comm. Algebra \textbf{23} (1995),
  no.~13, 4821--4831. \MR{1356104 (96k:13012)}

\bibitem{enochs:rha}
E.~E. Enochs and O.~M.~G. Jenda, \emph{Relative homological algebra}, de
  Gruyter Expositions in Mathematics, vol.~30, Walter de Gruyter \& Co.,
  Berlin, 2000. \MR{1753146 (2001h:16013)}

\bibitem{foxby:bcfm}
H.-B.\ Foxby, \emph{Bounded complexes of flat modules}, J. Pure Appl. Algebra
  \textbf{15} (1979), no.~2, 149--172. \MR{535182 (83c:13008)}

\bibitem{frankild:volh}
A.\ Frankild, \emph{Vanishing of local homology}, Math. Z. \textbf{244} (2003),
  no.~3, 615--630. \MR{1992028 (2004d:13027)}

\bibitem{frankild:rrhffd}
A.\ Frankild and S.\ Sather-Wagstaff, \emph{Reflexivity and ring homomorphisms
  of finite flat dimension}, Comm. Algebra \textbf{35} (2007), no.~2, 461--500.
  \MR{2294611}

\bibitem{greenlees:dfclh}
J.~P.~C. Greenlees and J.~P. May, \emph{Derived functors of {$I$}-adic
  completion and local homology}, J. Algebra \textbf{149} (1992), no.~2,
  438--453. \MR{1172439 (93h:13009)}

\bibitem{hartshorne:rad}
R.~Hartshorne, \emph{Residues and duality}, Lecture Notes in Mathematics, No.
  20, Springer-Verlag, Berlin, 1966. \MR{36 \#5145}

\bibitem{hartshorne:lc}
\bysame, \emph{Local cohomology}, A seminar given by A. Grothendieck, Harvard
  University, Fall, vol. 1961, Springer-Verlag, Berlin, 1967. \MR{0224620 (37
  \#219)}

\bibitem{hartshorne:adc}
\bysame, \emph{Affine duality and cofiniteness}, Invent. Math. \textbf{9}
  (1969/1970), 145--164. \MR{0257096 (41 \#1750)}

\bibitem{jensen:vl}
C.~U. Jensen, \emph{On the vanishing of
  {$\underset{\longleftarrow}{\lim}^{(i)}$}}, J. Algebra \textbf{15} (1970),
  151--166. \MR{0260839 (41 \#5460)}

\bibitem{kawasaki:ccma}
K.-i. Kawasaki, \emph{On a category of cofinite modules which is {A}belian},
  Math. Z. \textbf{269} (2011), no.~1-2, 587--608. \MR{2836085 (2012h:13026)}

\bibitem{kawasaki:ccc}
\bysame, \emph{On a characterization of cofinite complexes. {A}ddendum to
  ``{O}n a category of cofinite modules which is {A}belian''}, Math. Z.
  \textbf{275} (2013), no.~1-2, 641--646. \MR{3101824}

\bibitem{christensen:dcmca}
H.~Holm L.~W.~Christensen, H.-B.~Foxby, \emph{Derived category methods in
  commutative algebra}, in preparation.

\bibitem{lipman:llcd}
J.\ Lipman, \emph{Lectures on local cohomology and duality}, Local cohomology
  and its applications (Guanajuato, 1999), Lecture Notes in Pure and Appl.
  Math., vol. 226, Dekker, New York, 2002, pp.~39--89. \MR{1888195
  (2003b:13027)}

\bibitem{matlis:kcd}
E.~Matlis, \emph{The {K}oszul complex and duality}, Comm. Algebra \textbf{1}
  (1974), 87--144. \MR{0344241 (49 \#8980)}

\bibitem{matlis:hps}
\bysame, \emph{The higher properties of {$R$}-sequences}, J. Algebra
  \textbf{50} (1978), no.~1, 77--112. \MR{479882 (80a:13013)}

\bibitem{melkersson:mci}
Leif Melkersson, \emph{Modules cofinite with respect to an ideal}, J. Algebra
  \textbf{285} (2005), no.~2, 649--668. \MR{2125457 (2006i:13033)}

\bibitem{ogoma}
T.~Ogoma, \emph{Cohen {M}acaulay factorial domain is not necessarily
  {G}orenstein}, Mem. Fac. Sci. K\^ochi Univ. Ser. A Math. \textbf{3} (1982),
  65--74. \MR{643928 (83e:13026)}

\bibitem{yekutieli:hct}
M.~Porta, L.~Shaul, and A.~Yekutieli, \emph{On the homology of completion and
  torsion}, Algebr. Represent. Theory \textbf{17} (2014), no.~1, 31--67.
  \MR{3160712}

\bibitem{yekutieli:ccc}
\bysame, \emph{Cohomologically cofinite complexes}, Comm. Algebra \textbf{43}
  (2015), no.~2, 597--615. \MR{3274024}

\bibitem{yekutieli:ehct}
\bysame, \emph{Erratum to: {O}n the {H}omology of {C}ompletion and {T}orsion},
  Algebr. Represent. Theory \textbf{18} (2015), no.~5, 1401--1405. \MR{3422477}

\bibitem{raynaud:cpptpm}
M.\ Raynaud and L.\ Gruson, \emph{Crit\`eres de platitude et de projectivit\'e.
  {T}echniques de ``platification'' d'un module}, Invent. Math. \textbf{13}
  (1971), 1--89. \MR{0308104 (46 \#7219)}

\bibitem{sather:afbha}
S.~Sather-Wagstaff and R.~Wicklein, \emph{Adic finiteness: Bounding homology
  and applications}, preprint (2016), \texttt{arxiv:1602.03225}.

\bibitem{sather:afc}
\bysame, \emph{Adic {F}oxby classes}, preprint (2016),
  \texttt{arxiv:1602.03227}.

\bibitem{sather:asc}
\bysame, \emph{Adic semidualizing complexes}, preprint (2015),
  \texttt{arxiv:1506.07052}.

\bibitem{sather:afcc}
\bysame, \emph{Adically finite chain complexes}, preprint (2016),
  \texttt{arxiv:1602.03224}.

\bibitem{sather:scc}
\bysame, \emph{Support and adic finiteness for complexes}, Comm. Algebra, to
  appear, \texttt{arXiv:1401.6925}.

\bibitem{schoutens:lfccm}
H.~Schoutens, \emph{A local flatness criterion for complete modules}, Comm.
  Algebra \textbf{35} (2007), no.~1, 289--311. \MR{2287572 (2007h:13014)}

\bibitem{shaul:ard}
L.~Shaul, \emph{Adic reduction to the diagonal and a relation between
  cofiniteness and derived completion}, preprint (2016),
  \texttt{arxiv:1602.03874}.

\bibitem{shaul:hccac}
\bysame, \emph{Hochschild cohomology commutes with adic completion}, preprint
  (2015), \texttt{arxiv:1505.04172}.

\bibitem{simon:shpcm}
A.-M. Simon, \emph{Some homological properties of complete modules}, Math.
  Proc. Cambridge Philos. Soc. \textbf{108} (1990), no.~2, 231--246.
  \MR{1074711 (91k:13008)}

\bibitem{verdier:cd}
J.-L.\ Verdier, \emph{Cat\'{e}gories d\'{e}riv\'{e}es}, SGA 4$\frac{1}{2}$,
  Springer-Verlag, Berlin, 1977, Lecture Notes in Mathematics, Vol. 569,
  pp.~262--311. \MR{57 \#3132}

\bibitem{verdier:1}
\bysame, \emph{Des cat\'egories d\'eriv\'ees des cat\'egories ab\'eliennes},
  Ast\'erisque (1996), no.~239, xii+253 pp. (1997), With a preface by Luc
  Illusie, Edited and with a note by Georges Maltsiniotis. \MR{98c:18007}

\bibitem{yekutieli:fcigm}
A.~Yekutieli, \emph{On flatness and completion for infinitely generated modules
  over {N}oetherian rings}, Comm. Algebra \textbf{39} (2011), no.~11,
  4221--4245. \MR{2855123 (2012k:13058)}

\bibitem{yekutieli:sccmc}
\bysame, \emph{A separated cohomologically complete module is complete}, Comm.
  Algebra \textbf{43} (2015), no.~2, 616--622. \MR{3274025}

\end{thebibliography}
\providecommand{\bysame}{\leavevmode\hbox to3em{\hrulefill}\thinspace}
\providecommand{\MR}{\relax\ifhmode\unskip\space\fi MR }
\providecommand{\MRhref}[2]{%
  \href{http://www.ams.org/mathscinet-getitem?mr=#1}{#2}
}
\providecommand{\href}[2]{#2}

\end{document}